\documentclass[11pt]{article}
\usepackage{color,float,soul,graphicx}%
\usepackage{multirow,enumerate,enumitem}
\usepackage{url,cite,fullpage,fancyhdr,diffcoeff}
\usepackage{amsmath,amssymb,amsfonts,amsthm,mathrsfs,amscd}%
\usepackage{xcolor,textcomp,manyfoot}%
\usepackage{booktabs}%
\usepackage{algorithm,algorithmicx}%
\usepackage{algpseudocode,listings}%

\newcommand{\menge}[2]{\big\{{#1}~\big |~{#2}\big\}}

\newcommand{\fenv}[1]%
{\ensuremath{\,\overrightarrow{\operatorname{env}}_{#1}}}
\newcommand{\benv}[1]%
{\ensuremath{\,\overleftarrow{\operatorname{env}}_{#1}}}

\newcommand{\scal}[2]{\left\langle{#1},{#2}  \right\rangle}
\renewcommand\theenumi{(\roman{enumi})}

\theoremstyle{thmstyleone}%
\theoremstyle{thmstyletwo}%
\theoremstyle{thmstylethree}%
\newtheorem{theorem}{Theorem}
\newtheorem{lemma}[theorem]{Lemma}
\newtheorem{corollary}[theorem]{Corollary}

\usepackage{cite}
\theoremstyle{definition}

\newcommand{\la}{\langle}
\newcommand{\ra}{\rangle}
\newcommand{\Gph}{{\rm\textbf{Gph}\,}}
\newcommand{\dom}{{\rm\textbf{dom}\,}}

\newcommand{\nexto}{\kern -0.54em}

\newcommand{\dZ}{{\cal Z \kern -0.7em Z}}
\newcommand{\dC}{{\rm\hbox{C \kern-0.8em\raise0.2ex\hbox{\vrule height5.4pt width0.7pt}}}}
\newcommand{\dQ}{{\rm\hbox{Q \kern-0.85em\raise0.25ex\hbox{\vrule height5.4pt width0.7pt}}}}

\usepackage{epsfig}
\usepackage{latexsym}

\newcommand{\NN}{\mathbb{N}}
\newcommand{\HH}{\mathcal{H}}
\newcommand{\JJ}{\mathcal{J}}

\newcommand{\RR}{\mathbb{R}}

\newcommand{\ZZ}{\mathbb{Z}}

\date{}
\begin{document}
\title{Nonmonotone Spectral Analysis for Variational Inclusions 
}
\author{Oday Hazaimah
\footnote{E-mail: {\tt odayh982@yahoo.com}. https://orcid.org/0009-0000-8984-2500.  St. Louis, MO, USA}}
%

\maketitle

\begin{abstract}
Gradient descent algorithms perform well in convex optimization but can get tied for finding local minima in non-convex optimization. 
A robust method that combines a spectral approach with nonmonotone line search strategy for solving variational inclusion problems is proposed. 
Spectral properties using eigenvalues information are used for accelerating the convergence. Nonmonotonic behaviour is exhibited to relax descent property and escape local minima. 
Nonmonotone spectral conditions leverage adaptive search directions and global convergence for the proposed spectral subgradient algorithm. 
\\

\medskip 

\end{abstract}

\section{Introduction}
Variational inclusions are general mathematical frameworks arising naturally in many applied fields, such as finance, control theory, mechanics and operations research, see, for instance, \cite{Cavazzuti, Giannessi, Kinder}.
Variational inclusions expand many mathematical models including:
Nonlinear equations (solving $F(x)=0,$ where $F$ is a single-valued operator), variational inequalities, equilibrium problems and optimization (where the operator is the gradient of the objective function). 
One of the most known key component in treating variational inclusions is the set-valued operator, in which the operator assigns to each point in the domain a set of values, not just a single value. This is common in nonsmooth optimization, where the operator might represent the subdifferential of a nonsmooth function.
Variational inclusions have been studied numerously based on monotonic and nonmonotonic feedbacks span from fixed point iterations (use fixed point theorems to establish the existence of the solution), iterative methods (projection and gradient methods \cite{Bello-Hazaimah,Loreto-Euro}), regularization methods (for nonsmooth or nonconvex scenarios such as the popular proximal point algorithm), hybrid methods (combine monotone and nonmonotone), stochastic behaviours in the selection of step sizes, and splitting methods (such as the Douglas-Rachford algorithm \cite{D-R} where the idea is to decompose the original problem into simpler subproblems that can be solved more easily \cite{Bello-Hazaimah,tseng}).  
The main problem that we are interested in, is the variational inclusion problem which is pertaining to finding a point $x\in\HH$ such that:  
\begin{equation}\label{variational-inclusion}
0\in Tx,
\end{equation}
where $T$ is a set-valued nonlinear map and $\HH$ is an infinite-dimensional Hilbert space. As smooth set-valued maps are treating convex functions and singletons, $T$ in the setting of \eqref{variational-inclusion} can be nonsmooth to tackle nonconvex and (sometimes discontinuous) systems. One of the conventional and most effective techniques to solve \eqref{variational-inclusion} is the classical gradient (steepest) descent method in which it converges whenever the underlying operator is strongly monotone. The gradient descent method can be formulated as follows:
\begin{equation}\label{GD}
x_{k+1}=x_k-\lambda_kT(x_k),
\end{equation}
where $\lambda_k$ is called the \textit{step size} or, as in machine learning models, the \textit{learning rate}.  
Gradient descent algorithms can get lost in finding local minima when solving non-convex optimization problems. Spectral gradient methods, proposed originally by Barzilai and Borwein \cite{BzBw} and later developed by Raydan for quadratic order rate of convergence, are used effectively for adaptive step-size lengths to escape local minima for non-convex situations. 
It is known that spectral gradient methods converge faster than the classical gradient descent method.
A nonmonotone strategy incorporates spectral methods to relaxing the descent condition through a nonmonotone globalization line search since it allows for temporary increases in the objective value and maintain an overall decreasing behaviour. The spectral step size \cite{BzBw} is defined by the following formula:
\begin{equation}\label{spectral-step}
\lambda_k=\frac{\|x_k-x_{k-1}\|^2}{\la T(x_k)-T(x_{k-1}),x_k-x_{k-1}\ra}\ .
\end{equation}
Ideally, the spectral step $\lambda_k$ is chosen to be near 1, and satisfy $0<\lambda_{\min}<\lambda_{\max}<\infty.$
One advantage of using the spectral step rather than the regular line search is the reduced computational effort. Moreover, the inverse of \eqref{spectral-step} is a Rayleigh quotient corresponding to the average Jacobian matrix. 
The main inclusion problem \eqref{variational-inclusion} generalizes many important mathematical formulations as special cases which is listed as follows.
\begin{description}
\item[(i) Nonlinear systems:] 
If the operator $T$ is a single-valued operator $F$ then problem \eqref{variational-inclusion} can be reduced to finding $x\in\HH$ such that 
$$F(x)=0,$$
where $F:\HH\to\HH$ is a continuously differentiable map and its Jacobian is not accessible. Spectral method with nonmonotone line search have been applied to this type of nonlinear equations by La Cruz and Raydan in \cite{LaCruz} using merit functions to balance optimality and feasibility conditions. 
\item[(ii) Optimization:]
If the operator $T$ is a gradient of a smooth function $f, \ i.e.,\ T=\nabla f$, then the variational inclusion \eqref{variational-inclusion} is basically the unconstrained optimization problem 
$$\min_{x\in\HH}f(x),$$ 
and it is equivalent to solve the nonlinear system $\nabla f(x)=0$ globally. Moreover, if $T$ is a subgradient of a nonsmooth function $g$, then problem \eqref{variational-inclusion} is precisely the monotone inclusion $0\in\partial g(x),$ where $\partial g$ is the subdifferential mapping (defined in the next section).
\item[(iii) Variational inequalities:]
In such variational problems where one seeks to find a vector $x\in\HH$, satisfying 
$$\la Tx,y-x\ra\geq0.$$
This problem is called the variational inequality. It was first studied by Stampacchia \cite{Stampacchia}, and later developed with applications in \cite{Kinder} and in a more general settings in \cite{newtrends-gvi}. When the constraint set in the variational inequality is dynamic and not fixed then wide formulation will be extended, that is, quasi-variational inequalities which have many applications in game theory where payoff or cost functions experience nonmonotonic reactions \cite{Bensoussan-etal,Cavazzuti}. Variational inequalities have long history tied with applying dynamical systems approaches to provide discrete-time algorithms for finding stationary points and minimizers for the objective function, see \cite{Dupuis}. 
\end{description}

In this work we make use of the spectral step size introduced by Barzilai and Borwein in \cite{BzBw}, and developed later for large-scale nonlinear systems of equations by La Cruz and Raydan \cite{LaCruz}. Studying the spectral information along with gradient directions has got much attention and incorporate valuable and efficient techniques for smooth/nonsmooh optimization problems (see, for instance, \cite{Loreto-Euro, Loreto-CTI, Loreto-numerical} and the references therein). Extending the work of La Cruz and Raydan \cite{LaCruz} to variational inclusions where set-valued mappings are essential components for both smooth and nonsmooth scenarios, is worth exploring to provide novel insights and wide frameworks to the theory of variational inclusions. To the best of our knowledge, this is the first to applying the spectral subgradient algorithm combining the nonmonotone criterion and eigenvalues information for nonsmooth set-valued mappings in variational inclusions. Our primary motivation is to expand applying the spectral step techniques to set-valued maps which cover smooth and nonsmooth applications

\section{Background}
Let $\HH$ be a real Hilbert space equipped with inner product $\la \cdot , \cdot \ra$ and induced norm $\|\cdot\|:=\sqrt{\la\cdot,\cdot\ra}$. Let $T:\HH\rightrightarrows\HH$ be a set-valued (multifunction) operator and its domain be denoted by $\dom(T) :=\{ x\in\HH; T(x) <\infty\}$. 
Define the graph of $T$ by $\Gph(T):=\menge{(x,u)\in\dom(T)\times\HH}{u\in Tx}$. 
We say that $T$ is \textit{monotone mapping} if $\scal{x-y}{u-v}\geq 0,$ \ for all $(x,u), (y,v)\in\Gph(T).$
Furthermore, $T$ is maximally monotone if there exists no monotone operator $T^{\prime}$ such that $\Gph(T^{\prime})$ properly contains $\Gph(T)$. 
The inverse of $T$ is the set-valued operator defined by $T^{-1}\colon u \mapsto \menge{x\in\HH}{ u\in T(x)}$. The set of all subgradient vectors at the point $x$ form the subdifferential mapping 
$\partial g(x):=\{u\in\HH\ ;\ g(y)\geq g(x)+\langle u,y-x\rangle, \;\forall y\in\HH\}$ which is an example of a maximal monotone operator. The generalized Jacobian of $T$ at $x$ denoted by $\JJ_T(x)$.

\subsection{Spectral methods}
Spectral methods usually utilize spectral information (such as eigenvalues) 
of the underlying operator to improve the solution convergence and steer the direction of the step process for solving large-scale linear and nonlinear problems. The significance characteristic of spectral methods is that applying the spectral step does not depend on the optimal function value.
In many cases, the step size $\lambda_k$ is determined based on the spectral ratio \eqref{spectral-step}.
Spectral step sizes help lower the computational cost in large-scale systems and incorporate inexact line searches. While gradient methods might fail in global minima, the spectral method can significantly accelerate the convergence. Spectral gradient methods known for their low cost computationally and their pairing with nonmonotone line searches.

\subsection{Nonmonotone line search}
Traditional monotone algorithms (such as gradient descent) exhibit a monotonically behaviour in the function value. 
However, with the absence of monotonicity, global convergence criteria that carry nonmonotonic strategies to escape local minima for general mappings might allow limited changes in complex nonconvex nonsmooth problems. Nonmonotone methods leveraging the spectral properties of the operators involved in nonconvex or nonsmooth optimization processes and therefore its advantages can be summarized in the following three lines: faster convergence, robustness and flexibility in designing the algorithm and the iteration process. 
Nonmonotone spectral gradient methods, as a powerful approaches, combine ideas from nonmonotone line search techniques and spectral gradient methods, to achieve reliable performance for high-dimensional complex variational inclusions. 
In optimal control theory, the dynamics might exhibit nonmonotonic trajectories between states and controls. 

One of the most well-known monotone line search strategies is the one introduced by Grippo \textit{et al}. \cite{Grippo} for twice continuously differentiable objective functions in unconstrained optimization by means of Newton's method. This nonmonotone step length, which can be viewed as a generalization of Armijo's rule, is calculated by the following maximum inequality 
\begin{equation}\label{NMLS}
f(x_{k+1})\leq\max_{0\leq j\leq M}f(x_{k-j})+\gamma\lambda_k\la d_k,\nabla f(x_k)\ra,
\end{equation}
where $M\in\ZZ_+$ and $\gamma>0.$
The nonmonotonic behaviour in the search is caused by the maximum term as it allows temporary increases in the optimal function values.
The authors in \cite{Loreto-numerical} studied numerically nonsmooth uncostrained optimization problems by applying the classical subgradient method and a nonmonotone linesearch with the spectral step such that the objective function assumed to be convex and continuously differentiable almost everywhere (i.e., differentiable everywhere except on sets with Lebesgue measure zero). Loreto \textit{et al.} in \cite{Loreto-CTI} combined the spectral choice of step length with two subdifferential-like schemes: the gradient sampling method and the simplex gradient method. The objective function was assumed to be continuously differentiable almost everywhere and it is often not differentiable at minimizers (equivalently, the objective function is locally Lipschitz continuous and differentiable on an open dense subset of $\RR^n$).

\subsection{Splitting methods}
A wide class of schemes called splitting methods revolves around the idea of decomposing the set-valued operator $T$ into two different operators $T=A+B$, where $A$ is single-valued (referred to as {\em forward step}) and $B$ is set-valued (referred to as {\em backward step}).
One of the most important classical splitting methods to find a zero of the sum $T=A+B$ is the so-called {\em forward-backward} splitting method which is given as follows:
\begin{equation}\label{F-B}
x_{k+1}:=J_{\lambda_k B}(x_k-\lambda_k Ax_k), \end{equation}
where $\lambda_k >0$ for all $k\in\NN$, and $J$ is the resolvent map of $B$. This iteration converges weakly when $A$ is $\beta$-cocoercive, i.e., 
$$\forall x,y\in \HH,\qquad \langle Ax-Ay,x-y\rangle\ge \beta\|Ax-Ay\|^2,$$
It is worth emphasizing that the cocoercivity assumption of an operator is a strictly stronger property than Lipschitz continuity. 
To relax the cocoercivity assumption, Tseng \cite{tseng} proposed a modification of the \textit{forward backward} splitting method, known as the {\em forward-backward-forward} splitting iteration.

\section{Methodology and Main Results}
The geometric motivation behind using spectral information is that it is between the minimum and the maximum eigenvalue of the average Jacobian matrix. Using the spectral step size \eqref{spectral-step} and the Mean-Value theorem of integration, we obtain 
\begin{equation}\label{Rayleigh}
\frac{T(x_k)-T(x_{k-1})}{x_k-x_{k-1}}=\int_0^1 \JJ((1-t)x_{k-1}+tx_k)dt.
\end{equation}
It is clear to see that the spectral step \eqref{spectral-step} is the inverse of a Rayleigh quotient of the average Jacobian matrix (\textit{i.e.}, the right hand side of \eqref{Rayleigh}).

\subsection{The spectral subgradient algorithm}
The process of combining together the spectral-step size, and the nonmonotone line search given in \eqref{NMLS} due to Grippo \textit{at al.} \cite{Grippo}, along with the globalization strategy \cite{LaCruz} yield to a novel algorithm for solving variational inclusions. The algorithm is called the \textit{spectral subgradient} or the \textit{nonsmooth spectral} which uses the search direction $d_k=\pm T_k$. It is defined as follows:

\begin{center}\label{SSG}
\fbox{\begin{minipage}[b]{\textwidth}

\noindent{\bf Spectral Subgradient (SSG) Algorithm.} 

\medskip

\noindent{\bf Step 0. (Initialization):}
Take
\begin{equation*}
\alpha_0\in\RR, \gamma>0, M\in\ZZ_+\quad \mbox{and} \quad x_0\in\RR^n.
\end{equation*}

\noindent {\bf Step 1. (Procedure for stopping the search):} Given $x_k$, 
$$\mbox{If}\quad\|T_k\|=0,\quad\mbox{or}\quad\frac{\la T_k,\mathcal{J}_k T_k\ra}{\|T_k\|^2}<\varepsilon\quad\mbox{stop}.$$

\noindent {\bf Step 2.} If $\alpha_k,\displaystyle\frac{1}{\alpha_k}\leq\varepsilon$,\quad then set $\alpha_k=\delta,$ where  $\delta\in[\varepsilon,1/\varepsilon]$ for a small $\varepsilon$ between 0 and 1.
\medskip 

\noindent {\bf Step 3. (Search direction):} Set $\mbox{sgn}_k:=\mbox{sgn}\la T_k,\mathcal{J}_k T_k\ra$, and $d_k=-\mbox{sgn}_kT_k.$
\medskip

\noindent {\bf Step 4. (Nonmonotone global line search):} Set $\lambda=\frac{1}{\alpha_k}.$ If 
\begin{equation}\label{nmgls}
f(x_k+\lambda d_k)\leq\max_{0\leq j\leq\min(k,M)}f(x_{k-j})+2\gamma\lambda\la d_k,\mathcal{J}_k T_k\ra
\end{equation}
then, set $\lambda_k=\lambda,\quad x_{k+1}=x_k+\lambda_kd_k$. 
\medskip 

\noindent {\bf Step 5.} Choose $\theta\in(0,1),$ set $\lambda=\theta\lambda.$ Repeat step 4.
\medskip 

\noindent {\bf Step 6. (Stopping criterion):} Set $\alpha_{k+1}=\mbox{sgn}_k\displaystyle\frac{\la d_k,T_{k+1}-T_k\ra}{\lambda_k\|d_k\|^2}.$ 

If $x_{k+1}=x_k$ then stop.
\end{minipage}}\end{center}
\vspace{1.5cm} 

\begin{lemma}
The spectral subgradient method is well defined.
\end{lemma}
\begin{proof}
Suppose that the algorithm does not stop, at any iteration, at step 1, then $d_k$ is a descent direction. For $\gamma>0$ and sufficiently small $\lambda>0$, since $f$ is continuous then the following inequality 
$$f(x_k+\lambda d_k)\leq\max_{0\leq j\leq\min(k,M)}f(x_{k-j})+2\gamma\lambda\la T_k,\mathcal{J}_k\ra d_k$$
holds for all $k\geq 0$. Hence, the algorithm does not run many iterations between steps 4 and 5, and jump quickly to the update rule $x_{k+1}$. 
\end{proof}

\textbf{Assumption 1:} Before proving the convergence we need to assume the following for the rest of the paper:
\begin{enumerate}
\item Let $\alpha_{k+1}$ be computed as in step 6, is positive.
\item The level set 
$$\Gamma_0=\{x\in\HH:0\leq T(x)\leq T(x_0)\}$$ 
is bounded and closed such that it contains the generated sequence $\{x_k\}$.
\item Every accumulation point $\bar{x}$ satisfies $0\in\partial f(\bar{x})$ or in general $0\in T(\bar{x})$.
\item The set-valued map $T$ is nonsmooth. 
\item $T$ is bounded; namely, there exists $G>0$ such that $\|T\|\leq G$.
\item The subdifferential $\partial f$ is Lipschitz continuous on $\Gamma_0$, \textit{i.e.}, there exists $L>0$ such that 
$$\|\partial f(x)-\partial f(y)\|\leq L\|x-y\|,$$
for all $x,y\in\Gamma_0.$ 
\item The generalized Jacobian matrix $\mathcal{J}(x)$ is non-singular for any $x\in\Gamma_0.$ 
\end{enumerate}
\medskip 

\begin{lemma}\label{norms-subdifferentials-bdd}
Let the sequence $\{x_k\}$ be generated by the algorithm, then by assumption 1, and for all $k$, there exist the constants $c_1,c_2,c_3>0$ such that $$\|d_k\|\leq c_1\|\partial f(x_k)\|,\ \|\partial f(x_k)\|\leq c_2\|d_k\|\quad \mbox{and}\quad\la T,\JJ\ra d_k\leq -c_3\|\partial f(x_k)\|^2.$$
\end{lemma}
\begin{proof}
The proof is very similar to the proof of Lemma 3.3 in \cite{LaCruz} with slightly changes related to the subdifferential $\partial f$ rather than the smooth term $\nabla f.$
\end{proof}
The main theorem in this paper concludes that the spectral subgradient algorithm for solving variational inclusions stops after finitely many iterations and confirms that the generated sequence converges independently of any choice for the initial point. 
\begin{theorem}\label{main-thm}
Consider assumption 1 is valid, then the spectral subgradient algorithm \ref{SSG} either stops after finitely many iterations, or the generated sequence $\{x_k\}$ converges such that 
$$\lim_{k\to\infty}\|T_k\|=0.$$
Moreover, the sequence $\{x_k\}$ contained in $\Gamma_0$ and every accumulation point satisfies the inclusion $0\in T(\bar{x})$. 
\end{theorem}
\begin{proof}
Assuming \textbf{Assumption 1} is true. Suppose on the contrary, that the spectral subgradient algorithm does not stop after finitely many iterations. Let the generated sequence $\{x_k\}$ by Algorithm \ref{SSG} has an accumulation point $\bar{x}$. Following similar arguments of the convergence analysis of theorem 3.4 in \cite{LaCruz}, we obtain that $\Gamma_0$ is closed set and therefore it is compact. Note that $\min(0,M)=0$, and 
$$0\leq\min(k,M)\leq\min(\min(k-1,M)+1,M),\quad\forall k\geq 1.$$
By using Lemma \ref{norms-subdifferentials-bdd}, there exist two positive numbers $c_1,c_3$ such that $$\|d_k\|\leq c_1\|\partial f(x_k)\|$$ 
and 
$$\la T_k,\JJ_kd_k\ra\leq-c_3\|\partial f(x_k)\|^2.$$ 
Furthermore, the spectral steps are bounded by two positive numbers, i.e., 
$$0<\min\{\varepsilon,\frac{1}{\delta}\}\leq \lambda_k\leq\max\{\frac{1}{\varepsilon},\frac{1}{\delta}\}.$$ 
Repeating similar arguments of the proof of the main theorem in \cite{Grippo}, we obtain that 
$$0\in 2\la T(\bar{x}),\JJ(\bar{x})\ra=\partial f(\bar{x}).$$
Since $\mathcal{J}(x)$ is nonsingular in $\Gamma_0$, it is clear to conclude that $0\in T(\bar{x}).$

Finally, the proof of the last statement follows from the first part of the main theorem in \cite{Grippo} with slightly changes related to the set-valued subdifferential operators which can be investigated in the same manner whether set-valued operators are smooth or not.
\end{proof}

The strong global convergence of the spectral subgradient algorithm, without stopping the search procedure in step 1, is guaranteed when the convex combination of $\JJ$ and its transpose is positive definite.

\begin{corollary}
Assume all conditions in Theorem \ref{main-thm} hold and $\la T_k,\mathcal{J}_k T_k\ra\geq\varepsilon\|T_k\|^2$. 
If the convex combination of the generalized Jacobian and its transpose 
$\JJ_{cc}(x)=(1-r)\JJ(x)+r\JJ^T(x)$ 
is positive definite, then the spectral subgradient algorithm, either stops after finitely many iterations, or the generated sequence $\{x_k\}$ converges such that 
$$\lim_{k\to\infty}\|T_k\|=0.$$
\end{corollary}
\begin{proof}
Without loss of generality, suppose that $\JJ_{cc}$ is positive definite matrix for all $x\in\Gamma_0$. Suppose that the generalized Jacobian matrix $\JJ_{cc}$ has a smallest eigenvalue $\lambda_{\min}(\JJ_{cc}(x))$ such that $$0<u_{\min}\leq\lambda_{\min}(\JJ_{cc}(x_k)),$$ 
due to the continuity and compactness of $\Gamma_0$. Hence, for all $k\geq 0$, we have 
$$0<u_{\min}\|T_k\|^2\leq\lambda_{\min}(\JJ_{cc}(x_k))\|T_k\|^2\leq\la T_k,\JJ_{cc}(x_k)T_k\ra=\la T_k,\JJ_kT_k\ra.$$
Applying Theorem \ref{main-thm} on the set-valued mapping $T$, it follows that $\displaystyle\lim_{k\to\infty}\|T_k\|=0.$
\end{proof}

\section{Conclusion}
Variational inclusions provide a broad and flexible framework for modeling a variety of complex nonconvex nonsmooth optimization and decision-making processes. They are natural extensions to classical variational inequalities and nonlinear equations by allowing set-valued operators as essential components.
In this paper we proposed a spectral subgradient method, combining a spectral step size to accelerate global convergence which perform better than classical steps in gradient descent methods, and a nonmonotone globalization line search strategy to escape local minima and relax the descent condition in every iteration but allow short increases in the search for the objective function value. 
Large-scale nonconvex optimization problems, such as training neural networks or image recovering, can benefit from the flexibility of nonmonotone spectral line searches.
\medskip 

Future research directions may focus on developing efficient algorithms with normalizing the subgradient term in either discrete-time or continuous-time gradient descent methods for the purpose of relaxing the magnitude of the gradient. Normalized gradient descent is a variation of the standard gradient descent where the gradient vector is normalized before updating the parameters; meaning the step size remains the same irrespective of the gradient magnitude. One useful advantage is to integrate nonconvex optimization landscapes in machine learning models. More advantages can be listed; improving convergence, robustness, reduces sensitivity of the learning rate. focuses on the direction of the gradient rather than its magnitude. 
In addition, expanding the spectral notion to derivative-free schemes for solving variational inclusions and applying nonmonotone line searches for complex optimization methods over manifolds.



\subsection*{Declarations} 
Disclosure statement
The author declares that there was no conflict of interest or competing interest. This work was done without any resource of funding. Data Availability is not applicable.



\begin{thebibliography}{9}

\bibitem{BzBw} Barzilai J., Borwein J. Two point step size gradient methods. J. Numer. Anal. 8, 141-148, 1988.

\bibitem{Bello-Hazaimah} Bello-Cruz Y., Hazaimah O. On the weak and strong convergence of modified forward-backward-half-forward splitting methods. Optimization Letters. (17) 3, 2022.

\bibitem{Bensoussan-etal} Bensoussan A., Goursat M., Lions J.L. Controle impulsionnel et inequations quasi-variationnelles. Acad. Sci. Paris Ser. A 276, 1279–1284, 1973. 



\bibitem{Cavazzuti} Cavazzuti E., Pappalardo M., Passacantando M. Nash equilibria, variational inequalities and dynamical systems. Journal of Optimization Theory and Applications. 114, 491–506, 2002. 

\bibitem{LaCruz} Cruz W.L., Raydan M. Nonmonotone spectral methods for large-scale nonlinear systems. Optimization Methods and Software, 18(5), 583–599, 2003. 

\bibitem{D-R} Douglas J., Rachford H. On the numerical solution of heat conduction problems in two or three space variables. { \it Transactions of the American Mathematical Society}. (82) 421-439, 1956.

\bibitem{Dupuis} Dupuis P., Nagurney A. Dynamical systems and variational inequalities. Annals of Operations Research. 44, 9-42, 1993.






\bibitem{Giannessi} Giannessi F., Maugeri A., Pardalos P. Equilibrium oroblems: Nonsmooth optimization and variational inequality models. Kluwer Academics Publishers, Holland. 2001.

\bibitem{Grippo} Grippo L., Lampariello F., Lucidi S. A nonmonotone line search technique for Newton’s method. SIAM J. Numer. Anal. 23, 707–716, 1986.

\bibitem{Kinder} Kinderlehrer D., Stampacchia G. An Introduction to Variational Inequalities and Their Applications. Society for Industrial and Applied Mathematics (SIAM). Philadelphia, (2000).

\bibitem{Loreto-Euro} Loreto M., Aponte H., Cores D., Raydan M. Nonsmooth Spectral gradient methods for unconstrained optimization. EURO Journal on Computational Optimization. 
DOI: 10.1007/s13675-017-0080-8, 2017.

\bibitem{Loreto-CTI} Loreto M., Humphries T., Raghavan C., Wu K., Kwak S. A new spectral conjugate subgradient method with application in computed tomography image reconstruction. 2024. https://arxiv.org/abs/2309.15266v2

\bibitem{Loreto-numerical} Loreto M., Xu Y., Kotval D. A numerical study of applying spectral-step subgradient method for solving nonsmooth unconstrained optimization problems. Computers \& Operations Research. (104) 90-97, 2019.


\bibitem{newtrends-gvi} Noor M.A., Noor K.I., Rassias M.T. New Trends in General Variational Inequalities. Acta Appl Math. (170) 981–1064, 2020.

\bibitem{tseng} Tseng P. A modified forward-backward splitting method for maximal monotone mappings. {\it SIAM on Journal Control Optimization\/}. (38) 431-446, 2000.

\bibitem{Stampacchia} Stampacchia G. Formes bilineaires coercitives sur les ensembles convexes. Comptes Rendus Acad. Sci. Paris. (258) 4413-4416, 1964.



\end{thebibliography}

\end{document}